\documentclass[twoside,11pt]{article}
\usepackage[utf8]{inputenc}
\usepackage[T1]{fontenc}
\usepackage{setspace}
\usepackage{comment}
\usepackage{fancyhdr}

\usepackage[paper=a4paper,left=25mm,right=25mm,top=25mm,bottom=25mm]{geometry}
\usepackage[paper=a4paper,left=25mm,right=25mm,top=25mm,bottom=25mm]{geometry}
\usepackage{tikz}
\usepackage{tikz-cd}
\usepackage{amsmath,times}
\usepackage{amsmath}
\usepackage{amssymb}
\usetikzlibrary{matrix,arrows}
\usepackage{float}
\usepackage{amsthm}
\usepackage{enumitem}
\usepackage{authblk}
\usepackage{xcolor}

\usepackage{hyperref}

\newtheorem{T}{Theorem}
\newtheorem{theorem}{Proposition}[section]
\newtheorem{lemma}[theorem]{Lemma}
\newtheorem{cor}[theorem]{Corollary}

\theoremstyle{definition}
\newtheorem*{T*}{Theorem}
\newtheorem{defi}[theorem]{Definition}

\newtheorem{example}[theorem]{Example}

\newtheorem{remark}[theorem]{Remark}

\newcommand{\N}{\mathbb{N}}
\newcommand{\Z}{\mathbb{Z}}
\newcommand{\R}{\mathbb{R}}

\DeclareMathOperator{\SL}{SL}
\DeclareMathOperator{\PSL}{PSL}

\DeclareMathOperator{\cl}{cl}
\DeclareMathOperator{\scl}{scl}

\DeclareMathOperator{\Aut}{Aut}
\DeclareMathOperator{\Out}{Out}
\DeclareMathOperator{\Inn}{Inn}
\DeclareMathOperator{\sgn}{sgn}
\DeclareMathOperator{\Acode}{\mbox{$A$}-code}
\DeclareMathOperator{\Bcode}{\mbox{$B$}-code}
\DeclareMathOperator{\Ccode}{\mbox{$C$}-code}
\DeclareMathOperator{\Zcode}{\mathbb{Z}-code}

\title{AUT-INVARIANT QUASIMORPHISMS ON FREE PRODUCTS}
\date{%
}
\author{BASTIEN KARLHOFER}

\begin{document}
\maketitle

\begin{abstract}
Let $G=A \ast B$ be a free product of freely indecomposable groups. We explicitly construct quasimorphisms on $G$ which are invariant with respect to all automorphisms of $G$. We also prove that the space of such quasimorphisms is infinite-dimensional whenever $G$ is not the infinite dihedral group. As an application we prove that an invariant analogue of stable commutator length recently introduced by Kawasaki and Kimura is non-trivial for these groups. 
\end{abstract}

\section{Introduction}

The study of quasimorphisms on a given group $G$ is an important branch of geometric group theory \cite{Calegari} with quasimorphisms sharing deep relationships with the underlying structure of the group $G$. For free groups $F_n$ the so called counting quasimorphisms originating from the work of Brooks in \cite{Brooks} yield a wide variety of examples. His ideas have been developed further by Calegari and Fujiwara who constructed unbounded quasimorphisms on non-elementary hyperbolic groups \cite{Fuji}.  For diffeomorphism groups of surfaces many important constructions are given in \cite{Gamb}. There are also numerous applications in symplectic geometry originating from work of Entov and Polterovich \cite{Entov}. Another fundamental paper on the geometry of quasimorphisms and central extensions is \cite{Barge}.

In this paper we construct unbounded Aut-invariant quasimorphisms on free products of groups. To achieve this we associate tuples of natural numbers we call codes to each element in a free product $G = A \ast B$. Inspired by Brooks counting quasimorphisms on free groups we then count these codes rather than actual elements of the free product and verify that this indeed yields quasimorphisms on $G$. We call them code quasimorphisms. We make use of an explicit description of the automorphism group of a free product found in \cite{Gilbert} to see in Proposition \ref{freeprod no Z prop} that our code quasimorphisms are unbounded and invariant with respect to all automorphisms of $G$ if $A$ and $B$ are not infinite cyclic. 

If one of the factors of $G=A \ast B$ is infinite cyclic, our code quasi-morphisms are not necessarily invariant under a specific class of automorphisms of $G$ which is called the class of transvections. So, we slightly adjust the way we count codes for infinite cyclic factors and call the resulting maps weighted code quasimorphisms. We show in Proposition \ref{freeprod with Z} that these are unbounded and invariant with respect to all automorphisms of $G$.

These two propositions together with an independent result for the free group on two generators from \cite[Theorem 2]{BraMarc} comprise the following result in Section 6 which is the main result of the paper.

\begin{T}\label{T1}
Let $G= A \ast B$ be the free product of two non-trivial freely indecomposable groups $A$ and $B$. Assume $G$ is not the infinite dihedral group. Then $G$ admits infinitely many linearly independent homogeneous Aut-invariant quasimorphisms, all of which vanish on single letters. 
\end{T}

The infinite dihedral group does not admit any unbounded quasimorphism since all its elements are conjugate to their inverses. As a corollary of our construction we immediately deduce the existence of stably unbounded Aut-invariant norms on free products of two factors. 

\begin{cor}\label{MainCor}
Let $G= A \ast B$ be the free product of two non-trivial freely indecomposable groups and assume $G$ is not the infinite dihedral group. Then there exists a stably unbounded Aut-invariant norm on $G$.
\end{cor}
 
We denote the Aut-invariant stable commutator length, which was recently introduced by Kawasaki and Kimura in \cite{KK}, by $\scl_{\Aut}$. As an application of our construction we prove the following result in Section 7.

\begin{T}\label{T2}
Let $G=A \ast B$ be a free product of freely indecomposable groups and assume that $G$ is not the infinite dihedral group. Then there always exist elements $g \in G$ with positive Aut-invariant stable commutator length $\scl_{\Aut}(g) > 0 $. 
\end{T}

\section{Preliminaries}

\begin{defi}
For any group $G$ we denote by $\Aut(G)$ the group of all automorphisms of $G$. Moreover, we denote the normal subgroup of inner automorphisms by $\Inn(G)$ and define the group of outer automorphisms of $G$ to be the quotient $\Out(G)=\Aut(G)/\Inn(G)$.
\end{defi}

\begin{defi}
Let $G$ be a group. A map $\psi \colon G \rightarrow \R$ is called a \textit{quasimorphism} if there exists a constant $D \geq 0$ such that 
\begin{align*}
|\psi(g)+\psi(h)-\psi(gh)| \leq D \hspace{1mm} \text{for all} \hspace{1mm} g,h \in G.
\end{align*}
The smallest number $D(\psi)$ with the above property is called the \textit{defect} of $\psi$.
We call a quasimorphism \textit{homogeneous} if it satisfies $\psi(g^n)=n \psi(g)$ for all $g \in G$ and all $n \in \Z$. If  $\psi(\varphi(g))=\psi(g)$ for all $g \in G$, $\varphi \in \Aut(G)$, then $\psi$ is called Aut-invariant.
\end{defi}

\begin{defi}
Let $\psi \colon G \to \R$ be a quasimorphism. Then the \textit{homogenisation} $\bar{\psi} \colon G \to \R$ of $\psi$ is defined by $\bar{\psi}(g)= \lim_{n \in \N} \frac{\psi(g^n)}{n}$ for all $g \in G$.
\end{defi}

\begin{lemma}[{\cite[p.18]{Calegari}}] \label{HomogenisationQM}
The homogenisation $\bar{\psi}$ of a quasimorphism $\psi \colon G \to \R$ is a homogeneous quasimorphism. Moreover, it satisfies $|\bar{\psi}(g)- \psi(g)| \leq D(\psi)$ for any $g \in G$. 
\end{lemma}

\begin{defi}
A function $\nu \colon G \to \R$ satisfying for all $g, h \in G$:
\begin{itemize}
\item $\nu(g) \geq 0$, 
\item $\nu(g)=0$ if and only if $g=1$, 
\item $\nu(gh) \leq \nu(g)+ \nu(h)$,
\end{itemize}
is called a \textit{norm} on $G$. If in addition for all $g \in G$ and $\varphi \in \Aut(G)$ it satisfies 
\begin{itemize}
\item $\nu(\varphi(g))= \nu(g)$, 
\end{itemize}
then it is called \textit{Aut-invariant}. The supremum $\nu(G)= \sup \{\nu(g) \mid g \in G \}$ is called the \textit{diameter} of the norm $\nu$. If $\nu(G)= \infty$, then $\nu$ is called \textit{unbounded}.  If there exists $g \in G$ such that $\lim_{n \to \infty} \nu(g^n)= \infty$, then $\nu$ is called \textit{stably unbounded}.
\end{defi}

\begin{example}
Let $G$ be a group together with a generating set $S$. The \textit{word norm} generated by $S$ is the norm on $G$ defined by 
\begin{align*}
\nu_S(g)= \min \{ n \mid g= s_1 \cdots s_n \hspace{1mm} \text{where} \hspace{1mm} n \in \N \hspace{1mm} \text{and} \hspace{1mm} s_i \in S \hspace{1mm} \text{for all} \hspace{1mm} i \}.
\end{align*}
If we assume additionally that the set $S$ is invariant under $\Aut(G)$ then $\nu_S$ is Aut-invariant. 
\end{example}

\begin{lemma} \label{Autqm implies Autnorm}
Let $\psi \colon G \to R$ be an Aut-invariant quasimorphism with unbounded image, but bounded on a generating set $S$ of $G$. Then there exists a stably unbounded Aut-invariant norm on $G$.
\end{lemma}

\begin{proof}
By Lemma \ref{HomogenisationQM} we can assume that $\psi$ is homogeneous. The word norm $\|.\|_{\bar{S}}$ on $G$ associated to the generating set $\bar{S}=\{\varphi(s) \mid s \in S, \varphi \in \Aut(G) \}$ is clearly Aut-invariant. Let $K$ be a positive bound for the absolute value of $\psi$ on $S$. Write $g \in G$ as a product $g=\varphi_1(s_1) \cdots \varphi_n(s_n)$ for some $n \in \N$ where $s_i \in S$ and $\varphi_i \in \Aut(G)$ for all $i$. The calculation 
\begin{align*}
| \psi(g)| = | \psi(\varphi_1(s_1) \cdots \varphi_n(s_n)) | \leq |\psi(\varphi_1(s_1))|+ \dots + |\psi(\varphi_n(s_n))| \hspace{-0.1mm}+(n-1)D(\psi) \leq n(K + D(\psi)) 
\end{align*}
shows that $\|g\|_{\bar{S}} \geq \frac{|\psi(g)|}{K+D(\psi)}$ for all $g \in G$. It follows that $\|g^k\|_{\bar{S}} \geq k \cdot \frac{|\psi(g)|}{K+D(\psi)}$  for all $k \in \N$, $g \in G$. Since $\psi$ does not vanish everywhere, $\|.\|_{\bar{S}}$ is a stably unbounded Aut-invariant norm on $G$. 
\end{proof}

The following examples illustrate that the converse of Lemma \ref{Autqm implies Autnorm} above is not true and finding unbounded Aut-invariant quasi-morphisms is much more difficult than finding unbounded Aut-invariant norms.

\begin{example} 
Let $\Sigma_\infty$ be the infinite symmetric group of finitely supported bijections of the natural numbers. The cardinality of the support defines an Aut-invariant norm of infinite diameter on $\Sigma_\infty$. However, any element $g \in \Sigma_\infty$ has finite order. Therefore, no Aut-invariant norm on $\Sigma_\infty$ is stably unbounded and any homogeneous quasi-morphism vanishes on all of $\Sigma_\infty$. Consequently, by Lemma \ref{HomogenisationQM} any quasimorphism on $\Sigma_\infty$ is bounded.  
\end{example}

\begin{example}
Let $G= \Z^k$ for $k \geq 1$. Since every $g \in G$ lies in the same $\Aut(G)$-orbit as $g^{-1}$, it follows that any homogeneous Aut-invariant quasimorphism vanishes on all of $G$. So any Aut-invariant quasimorphism is bounded on $G$. For $k=1$ the standard absolute value defines a stably unbounded Aut-invariant norm on $G$, whereas for $k \geq 2$ any Aut-invariant norm on $G$ has finite diameter. 
\end{example}

\begin{example} \label{Kleinbottle}
Let $G$ be the fundamental group of the Klein bottle $G= \Z \ast_{2\Z} \Z$. Let $a$ and $b$ be generators of the two infinite cyclic factors of $G$ in its above presentation. Consider the Aut-invariant word norm $\nu_S$ generated by $S= \{ \varphi(a^{\pm 1}), \varphi(b^{\pm 1}) \mid \varphi \in \Aut(G) \}$. To see that this norm is unbounded on $G$ we first note that commutator subgroup of $G$ is a characteristic subgroup and $G/[G,G]= \Z/2 \times \Z$, where $\Z/2$ is a characteristic subgroup again. Consequently, the projection map $p \colon G \to \Z$ sending $p(a)=p(b)=1$ maps the set $S$ to the Aut-invariant set $\{\pm 1\}$ in $\Z$, which generates a stably unbounded Aut-invariant norm on $\Z$. Therefore, the word norm $\nu_S$ generated by $S$ on $G$ is stably unbounded as well. 

However, there is no unbounded Aut-invariant quasimorphism on $G$. If $\psi$ was such a quasimorphism, it could be chosen to be homogeneous by Lemma \ref{HomogenisationQM}. Let $\varphi$ be the automorphism inverting the generators $a$ and $b$. Then $\varphi$ inverts the center $Z(G)=2 \Z$ as well. Hence, $\psi$ vanishes on $Z(G)$. Similarly, every element of $S=\{(ab)^n, (ba)^n, (ab)^n a , (ba)^n b \hspace{1mm}|\hspace{1mm} n \in \N \}$ belongs to the same Aut-orbit that its inverse belongs to. So $\psi$ vanishes on $S$ as well. However, every element $g \in G$ can be written as a product $=zs$ where $z \in Z(G)$ and $s \in S$. Therefore, $\psi$ is bounded on all of $G$.
\end{example}

\begin{defi}
Let $G= \ast_{i \in I} G_i$ be a free product of a family of groups $\{G_i\}_{i \in I}$ for some indexing set $I$. For each $i$ the factor $G_i$ is a subgroup of $G$ via the canonical inclusion. An element of $G$ that belongs to one of the factors is called a \textit{letter} of $G$. Any product of letters is called a \textit{word} in $G$. The product of any two letters belonging to the same factor in $G$ can be replaced by the letter that represents their product in that factor. Moreover, any identity letters appearing in a word can be omitted without changing the element the word represents in $G$.  Recall that any element $g \in G$ has a unique presentation as a word, where no two consecutive letters lie in the same factor and no identity letters appear. Such a word is called \textit{reduced}.
\end{defi}

\begin{lemma} \label{building quasimorphisms}
Let $I$ be a set of cardinality at least two. Let $G_i$ be a non-trivial group for all $i$ and $G=\ast_{i \in I} G_i$ be their free product. Let $\theta \colon G \to \R$ be a map whose absolute value is bounded on all letters of $G$ by a constant $B \geq 0$. Assume that there exists a constant $D \geq 0$ such that
\begin{align*}
| \theta(w_1 w_2 ) - \theta (w_1) - \theta (w_2) | \leq D
\end{align*}
holds for all reduced words $w_1, w_2$ for which their product $w_1 w_2$ is a reduced word. Then the map $f \colon G \to \R$ defined by $f(w)= \theta(w) - \theta(w^{-1})$ defines a quasimorphism of defect at most $12D+6B$, which is bounded on all letters by $2B$.  
\end{lemma}

\begin{proof}
Any element in $G$ can be represented by a reduced word. So let $w_1$, $w_2$ be reduced words. The word given by their product $w_1 w_2$ is reduced if and only if the last letter from $w_1$ belongs to a factor different from the one that the first letter of $w_2$ belongs to. Indeed, otherwise those two letters could be multiplied in their common factor and replaced by their product to shorten the number of letters appearing in the expression. 

In order to bring $w_1 \cdot w_2$ to its reduced form we first perform all cancellations which form a word we call $c$. After all cancellations have taken place the final potential reduction is to possibly replace a non-trivial product of two letters $b$ and $d$ belonging to the same factor by a non-trivial letter $x$ representing their product in that factor. Therefore, we have two cases.
\begin{itemize}
\item The reduced presentations of $w_1$ and $w_2$ are given by $w_1= ac$, $w_2=c^{-1}e$ and $ae$ is the reduced presentation for $w_1 \cdot w_2$.
\item The reduced presentations of $w_1$ and $w_2$ are given by $w_1= abc$, $w_2 = c^{-1} de$, where $b$ and $d$ are letters belonging to the same factor. The reduced presentation of $w_1 \cdot w_2$ is given by $a x e$, where $x=bd$ is the letter representing the non-trivial product of $b$ and $d$.
\end{itemize}
  We calculate for the second case that
\begin{align*}
 |f(w_1 w_2) - f(w_1) & -f(w_2)| =   |f(axe) - f(abc) -f(c^{-1} de)| \\
= &  \left | \theta(axe) - \theta(e^{-1} x^{-1} a^{-1}) - \theta(abc) + \theta(c^{-1} b^{-1} a^{-1}) - \theta(c^{-1} de) + \theta (e^{-1} d^{-1} c) \right | \\
 \leq  & | \theta(a) + \theta(x) + \theta(e) - \theta(e^{-1}) - \theta(x^{-1}) - \theta(a^{-1}) - \theta(a) - \theta(b) - \theta(c) + \theta(c^{-1})\\
 &  + \theta(b^{-1}) + \theta(a^{-1}) - \theta(c^{-1}) - \theta(d) - \theta(e) + \theta(e^{-1}) + \theta(d^{-1}) + \theta(c)| + 12D\\
= &  | \theta(x) - \theta(x^{-1}) - \theta(b) + \theta(b^{-1}) - \theta(d) + \theta(d^{-1})| + 12D \\
\leq & 6B + 12D.
\end{align*}
The first case follows analogously. Since $w_1$, $w_2$ were arbitrary reduced words and every element of $G$ can be written in its reduced form, $f$ is a quasimorphism of defect at most $6B+ 12D$. Since $\theta$ is bounded on all letters by $B$, so is $f$ by $2B$. 
\end{proof}

\section{Aut-invariant quasimorphisms}

\begin{defi}
We call a group $G$ \textit{freely indecomposable} if $G$ is non-trivial and not isomorphic to any free product of the form $G_1 \ast G_2$ where $G_1$, $G_2$ are non-trivial groups.
\end{defi}

Any free product of non-trivial groups has trivial center and contains elements of infinite order. So every abelian group and every finite group is freely indecomposable.

\begin{lemma}\label{homogenisation Aut invariant}
Let $\psi \colon G \to \R$ be a quasimorphism. Let $\{\varphi_i \}_{i \in I}$ be a set of representatives for the elements of $\Out(G)$. If $\psi$ is invariant under $\varphi_i$ for all $i$, then its homogenisation $\bar{\psi} \colon G \to \R$ is invariant under all automorphisms of $G$. 
\end{lemma}

\begin{proof}
The homogenisation $\bar{\psi}$ is constant on conjugacy classes \cite[p.19]{Calegari}.
By definition $\bar{\psi}$ is also invariant under the collection $\{ \varphi_i \}_{i \in I}$, since $\psi$ is. The result follows since any element $\varphi \in \Aut(G)$ can be written as the composition of some $\varphi_j$ with a conjugation.
\end{proof}

Consider the free product $G= G_1 \ast G_2$ where $G_i$ is freely indecomposable for $i=1,2$. Following the exposition in \cite[p.116]{Gilbert} based on results in \cite{Fouxe1} and \cite{Fouxe2} the automorphism group $\Aut(G_1 \ast G_2)$ is generated by the following types of automorphisms (1.-3.) if neither $G_1$ nor $G_2$ is infinite cyclic: 
\begin{enumerate}
\item Elements from $\Aut(G_1)$ and $\Aut(G_2)$ give rise to automorphisms of $G_1 \ast G_2$. These are called \textit{factor automorphisms}.
\item Let $g \in G_i$ for some $i \in \{ 1,2 \}$. Define the map $p_g \colon G \to G$ to be conjugation by $g$ on the letters of $G_j$ for $j \neq i$ and to be the identity on all letters from the group $G_i$. This definition gives rise to an automorphism of $G$ which is called a \textit{partial conjugation}.
\item If $G_1 \cong G_2$ are isomorphic, interchanging the two factors is an automorphism of $G$. Such an automorphism is called a \textit{swap automorphism}.   
\end{enumerate}
If $G_1 \cong \Z$ is infinite cyclic and the freely indecomposable group $G_2$ is not, then $\Aut(G_1 \ast G_2)$ is generated by the above automorphisms together with the following additional type of automorphisms: 
\begin{enumerate}[resume]
\item Let $s$ be a generator of $G_1$ and let $a \in G_2$ be any element. Then a \textit{transvection} is the unique automorphism of $G_1 \ast G_2$ defined to be the identity on all letters from $G_2$ and maps $s \to as$ or $s \to sa$. 
\end{enumerate}

Following the above description of the group of automorphisms of a free product of two factors we obtain:

\begin{lemma}\label{outrepresentatives}
Let $G_1$, $G_2$ be freely indecomposable groups such that $G_2$ is not infinite cyclic. Then the outer automorphism group of their free product $\Out(G_1 \ast G_2)$ is generated by the images of $\Aut(G_1)$, $\Aut(G_2)$ in $\Out(G_1 \ast G_2)$ together with a swap automorphism if $G_1 \cong G_2$ and the transvections if $G_1 \cong \Z$.
\end{lemma}

\begin{proof}
By the universal property of the free product of two groups any automorphism is uniquely determined by its image on single letters. Let $h \in G_1$ and denote conjugation by $h^{-1}$ on all of $G$ by $c_h$. Then
\begin{align*}
(c_h \circ p_h)(g) = 
\begin{cases}
 h^{-1} g h & \text{if } \hspace{1mm} g \in G_1, \\
 g & \text{if } \hspace{1mm} g \in G_2.
\end{cases}
\end{align*}
Thus, $p_h$ and the factor automorphism given by conjugation by $h^{-1}$ on $G_1$ represent the same element in $\Out(G)$. Similarly, in $\Out(G)$ partial conjugations on $G_1$ by elements from $G_2$ represent the same elements that factor automorphisms from $G_2$ do. Finally, any two choices of swap automorphism differ by a product of factor automorphisms.
\end{proof}

\begin{lemma} \label{Gen6}
Let $G$ be a group and $H \leq G$ be a characteristic subgroup with quotient projection $p \colon G \rightarrow G/H$. Then for any unbounded Aut-invariant quasimorphism $\psi \colon G/H \to \R$ the composition $\psi \circ p \colon G \to \R$ is 
an unbounded Aut-invariant quasimorphism on $G$. Moreover, linearly independent quasimorphisms on $G/H$ give rise to linearly independent quasimorphisms on $G$. 
\end{lemma}

\begin{proof}
Clearly, $\psi \circ p$ is a quasimorphism. The Aut-invariance of $\psi \circ p$ on $G$ follows from the Aut-invariance of $\psi$ on $G/H$ together with the fact that $H$ is characteristic. Finally, the statement about linear independence follows from the surjectivity of the projection to the quotient. 
\end{proof}

\section{Code quasimorphisms}

Recall that a \textit{tuple} always refers to a finite sequence and so all tuples are naturally ordered.

\begin{defi}
Let $A$ and $B$ be groups. Write a given element $ g \in A \ast B$ in its reduced form. We assign two tuples of non-zero natural number that we will call \textit{codes} as follows. Let $( a_1, \dots, a_k)$ be the tuple of letters from $A$ appearing in the reduced form of $g$. We call $( a_1, \dots, a_k)$ the \textit{A-tuple} of $g$. Then we count how often any one letter of $( a_1, \dots, a_k)$ appears consecutively. This yields a tuple of positive numbers $\Acode(g)=(n_1, n_2, \dots, n_r)$ which we call the \textit{A-code} of $g$. Similarly, we obtain the \textit{B-tuple}, which is the tuple of letters from $B$ appearing in the reduced form of $g$, and the \textit{B-code} of $g$, denoted $\Bcode(g)$, by counting consecutive appearances of letters in the \textit{B-tuple}.  

%In the same way we obtain a tuple of natural numbers $\Bcode(g)=(m_1, m_2, \dots, m_s)$, called the \textit{B-code} of $g$, coming from counting consecutive appearances in the \textit{B-tuple} $( b_1, \dots , b_\ell)$ of letters from $B$. 
\end{defi}

Note that $\Acode(g)$ and $\Bcode(g)$ might have very different length for elements $g \in A \ast B$ in general.

\begin{example}\label{examplecode}
Let $G=A \ast B$ where $A= \Z/5$ and $B$ is any group. Let $a \in A, b \in B$ be non-trivial elements. Consider $g=a^2bababa^4baba$. The $A$-tuple of $g$ is $(a^2,a,a,a^4,a,a)$ and therefore $\Acode(g)=(1, 2, 1, 2)$. However, the $B$-tuple of $g$ is $(b,b,b,b,b)$ and so $\Bcode(g)=(5)$.
\end{example}

\begin{remark} \label{Autinvariance code}
The code of any element $g \in A \ast B$ is clearly invariant under all factor automorphisms.
\end{remark}

The following lemma is immediate. 

\begin{lemma} \label{code counting mirror}
The $\Acode$ and $\Bcode$ of $g^{-1}$ are the reversed A- and $\Bcode$ of $g$ for any $g \in A \ast B$. That is, let $\Acode(g)=(n_1, \dots, n_k)$ and  $\Bcode(g)=(m_1, \dots, m_\ell)$, then $\Acode(g^{-1})=(n_k, \dots , n_1)$ and $\Bcode(g^{-1})=(m_\ell, \dots, m_1)$. \qed
\end{lemma}

In the spirit of Brooks counting quasimorphisms we will now define \textit{code quasimorphisms}, which are counting the occurrences of a string of natural numbers in the $\Acode$ and $\Bcode$ associated to an element in the free product $A \ast B$. 

\begin{defi}[Code quasimorphisms]
Let $k \geq 1$ and let $z=(n_1, \dots ,n_k)$ be a tuple of non-zero natural numbers $n_1, \dots ,n_k$ for some $k \in \N$. Let $C \in \{ A, B \}$. Define
$\theta^C_z \colon A \ast B \to \Z_{\geq 0}$ to count the maximal number of \textit{disjoint} appearances of $z$ as a tuple of consecutive numbers in the $\Ccode$ for all $g \in A \ast B$. Further, define the \textit{code quasimorphism} 
\begin{align*}
f^C_z \colon A \ast B \to \Z \hspace{5mm} \text{by} \hspace{5mm} f^C_z(g)= \theta^C_{z}(g) - \theta^C_{z}(g^{-1})
\end{align*}
for all $g \in A \ast B$. Note that $\theta^C_{z}(g^{-1})= \theta^C_{\bar{z}}(g)$ due to Lemma \ref{code counting mirror}, where $\bar{z}$ denotes the reversed tuple $(n_k, \dots, n_1)$. Consequently, $f^C_z(g)$ can also be written as $f^C_z(g)= \theta^C_{z}(g) - \theta^C_{\bar{z}}(g)$ for all $g \in G$.
\end{defi}

\begin{example}
Let $G=\Z/5 \ast B$ and $g=a^2bababa^4baba$ for non-trivial $a \in A, b \in B$ as in Example \ref{examplecode}. For $z=(1,2)$ we calculate $\theta^A_z(g)=2$ and $\theta^A_{z}(g^{-1})= \theta^A_{\bar{z}}(g)=1$ and so $f^A_z(g)=2-1=1$.
\end{example}

\begin{example}
Let $G=\Z/5 \ast B$ and $g=a^4bababa^3bababa^3$ for non-trivial $a \in A, b \in B$. In this case $\Acode(g)=(1,2,1,2,1)$. Then $\theta^A_z(g)=1$ for $z=(1,2,1)$ since we only count disjoint occurrences. Similarly, $\theta^A_{z}(g^{-1})= \theta^A_{\bar{z}}(g)=1$ and so $f^A_z(g)=0$.
\end{example}

\begin{lemma}\label{codeqm}
Let $A$, $B$ be non-trivial groups and let $C \in \{A, B \}$. For a non-empty tuple of non-zero natural numbers $z$ the map $f^C_z \colon A \ast B \to \Z$ defines a quasimorphism that is bounded on letters and invariant with respect to all factor automorphisms. Moreover, $D(f^C_z) \leq 30$. 
\end{lemma}

\begin{proof}
We want to apply Lemma \ref{building quasimorphisms} to deduce that $f^C_z$ is a quasimorphism. Clearly, $|\theta^C_z(x)| \leq 1 $ for all letters $x \in A \ast B$ and all $z$. Let $w_1, w_2$ be reduced words representing elements in $A \ast B$ such that their product $w_1 w_2$ is reduced. That is, the last letter of $w_1$ and the first letter of $w_2$ belong to different factors. Without loss of generality we can assume $C=A$. Let $\Acode(w_1)=(n_1, \dots, n_k)$ and  $\Acode(w_2)=(m_1, \dots , m_\ell)$. Let $x$ be the last letter from $A$ in $w_1$ and let $y$ be the first letter from $A$ in $w_2$. Then
\begin{align*}
\Acode(w_1 w_2) = 
\begin{cases}
(n_1, \dots, n_k, m_1 , \dots , m_\ell) & \text{if} \hspace{1mm} x \neq  y, \\
(n_1, \dots n_{k-1}, n_k+m_1, m_2, \dots, m_\ell)  & \text{if} \hspace{1mm} x=y.
\end{cases}
\end{align*}

If $x \neq y$, then $\theta^C_z(w_1 w_2 ) \in \{ \theta^C_z(w_1) + \theta^C_z(w_2), \theta^C_z(w_1) + \theta^C_z(w_2)+1 \}$ since at most one of the disjoint occurrences of $z$ can involve numbers that do not lie completely in the $\Acode$ of either $w_1$ or $w_2$. 

If $x=y$, then $\theta^C_z(w_1 w_2 ) \geq \theta^C_z(w_1) + \theta^C_z(w_2) - 2$ since $n_k$ and $m_1$ can each be contained in at most one occurrences of $z$ in the $\Acode$ of $w_1$ and $w_2$. Moreover, if an occurrence of $z$ in the $\Acode$ of $w_1 w_2$ involves $n_k + m_1$, then all other occurrences are fully contained in either the $\Acode$ of $w_1$ or $w_2$. Thus, $\theta^C_z(w_1 w_2 ) \leq \theta^C_z(w_1) + \theta^C_z(w_2) +1 $. 

In both cases we conclude 
\begin{align*}
| \theta^C_z(w_1 w_2 ) - \theta^C_z(w_1) - \theta^C_z(w_2)| \leq 2,
\end{align*}
and it follows from Lemma \ref{building quasimorphisms} that $f^C_z$ is a quasimorphism of defect $D(f^C_z) \leq 30$.

Moreover, by Remark \ref{Autinvariance code} the maps $\theta^C_z$ are invariant under all factor automorphisms of $A \ast B$. Consequently, $f^C_z= \theta^C_{z} - \theta^C_{\bar{z}}$ is invariant under factor automorphisms as well.
\end{proof} 

\begin{defi}
A tuple of non-zero natural numbers $z=(n_1, \dots, n_k)$  is called \textit{generic} if $\bar{z}$ does not appear as a tuple of $k$ adjacent numbers in $z^2=(n_1, \dots, n_k, n_1, \dots n_k)$. 
\end{defi}

\begin{example}
Let $z=(n_1, \dots, n_k)$. If $k \leq 2$, $z$ is not generic. If $k \geq 3$ and the $n_i$ are pairwise distinct, then $z$ is generic. E.g. for $z=(1,2,3)$ we have $\bar{z}=(3,2,1)$ does not appear in $z^2=(1,2,3,1,2,3)$. 
\end{example}

\begin{theorem} \label{freeprod no Z prop}
Let $A \ast B$ be a free product of two freely indecomposable groups $A$ and $B$, neither of which is infinite cyclic. Then for any generic tuple of natural numbers $z$ the following holds:
\begin{enumerate}
\item if $A \ncong B$ and $C \in \{A,B \}$ is such that $C \ncong \Z/2$, then the homogenisation $\bar{f}^C_z$ of the quasimorphism $f^C_z$ is an unbounded Aut-invariant quasimorphism on $A \ast B$;
\item if $A \cong B \ncong \Z/2$, then the sum $\bar{f}^A_z + \bar{f}^B_z$ is an unbounded Aut-invariant quasimorphism on $A \ast B$. 
\end{enumerate}
In both cases the space of homogeneous Aut-invariant quasimorphisms on $A \ast B$ that vanish on letters has infinite dimension.
\end{theorem}

\begin{proof}
First, consider the case $A \ncong B$. Since $A \ast B$ is not the infinite dihedral group, at least one of the factors is not isomorphic to $\Z/2$. Without loss of generality we assume $A \ncong \Z/2$. Let $z$ be generic. By Lemma \ref{codeqm} the map $f^A_z$ defines a quasimorphism invariant under all factor automorphisms. According to Lemma \ref{outrepresentatives} this means that $f^A_z$ is invariant under a full set of representatives for $\Out(A \ast B)$. Therefore, the homogenisation $\bar{f}^A_z$ is invariant under all automorphisms of $A \ast B$ by Lemma \ref{homogenisation Aut invariant}. It remains to check that $\bar{f}^A_z$ is unbounded, which is equivalent to checking that $f^A_z$ itself is unbounded by Lemma \ref{HomogenisationQM}. 

Since $A \ncong \Z/2$, it satisfies $|A| \geq 3$ and we can choose two distinct non-trivial elements $a_1,a_2 \in A$. Furthermore, choose a non-trivial element $b \in B$. Let $z = (n_1, \dots, n_k)$ and choose $m \in \N$ to be non-zero and distinct from all $n_i \in \N$. We set 
\begin{align*}
w_0=(a_1b)^{n_1} (a_2b)^{n_2} (a_1 b)^{n_3} (a_2b)^{n_4} \dots (a_sb)^{n_k},
\end{align*}
where $s=1$ if $k$ is odd and $s=2$ if $k$ is even. Set 
\begin{align*}
w= 
\begin{cases}
w_0 & \text{if } k \text{ is even},\\
w_0 (a_2b)^m & \text{if } k \text{ is odd}.
\end{cases}
\end{align*}
The $\Acode$ of $w$ is given by 
\begin{align*}
\Acode(w)= 
\begin{cases}
(n_1, \dots, n_k)=z & \text{if } k \text{ is even},\\
(n_1, \dots, n_k, m ) = (z,m) & \text{if } k \text{ is odd}.
\end{cases}
\end{align*}

Since $w$ starts and ends with letters from different groups, the reduced expression of $w^\ell$ is the $\ell$-fold product of the word $w$ for all $\ell \in \N$. Moreover, because the first letter from $A$ in $w$ is $a_1$ and the last letter from $A$ is $a_2$, the $\Acode$ of $w^\ell$ is 
\begin{align*}
\Acode(w^\ell)= 
\begin{cases}
(z, z, \dots, z) & \text{if } k \text{ is even},\\
(z, m, z, m, \dots, z, m) & \text{if } k \text{ is odd}.
\end{cases}
\end{align*}

Since $m$ is distinct from all $n_i$, $m$ can never appear in any occurrence of $z$ or $\bar{z}$ in the $\Acode$ of $w^\ell$. So $\theta^A_z(w^\ell)= \ell$, whereas $\theta^A_{\bar{z}}(w^\ell)=0$ since $z$ is generic. Consequently, 
\begin{align*}
f^A_z(w^\ell)=\theta^A_z(w^\ell) - \theta^A_{\bar{z}}(w^\ell)=\ell,
\end{align*}
which shows that $f^A_z$ is unbounded.

Second, consider the case $A \cong B$ and fix a choice of isomorphism. Let $z$ be generic. It holds that $|A|=|B| \geq 3$ since $A \ast B$ is not the infinite dihedral group. 
Consider the swap isomorphism $s$ interchanging the factors $A$ and $B$, where we use the fixed isomorphism from before to identify $A$ and $B$ with each other. Then the application of $s$ to any element $g$ interchanges the $\Acode$ and $\Bcode$ of $g$ with each other. This implies that the sum $\theta^A_z + \theta^B_z$ is invariant under $s$ and consequently the sum $f^A_z + f^B_z$ is invariant under $s$ as well.  Again, by Lemma \ref{codeqm} $f^A_z$ and $f^B_z$ define quasimorphisms invariant under all factor automorphisms and so does their sum $f^A_z + f^B_z$. According to Lemma \ref{outrepresentatives} this means that $f^A_z + f^B_z$ is invariant under a full set of representatives for $\Out(A \ast B)$. Again, by Lemma \ref{homogenisation Aut invariant} we see that the homogenisation $\bar{f}^A_z + \bar{f}^B_z$ is invariant under all automorphisms of $A \ast B$. It remains to verify unboundedness. 

For this let $a_1,a_2 \in A$ and $b_1,b_2 \in B$ be non-trivial such that $a_1 \neq a_2$ and $b_1 \neq b_2$. Pick a non-zero number $m \in \N$ distinct from all $n_i \in \N$, where $z=(n_1, \dots, n_k)$. As before, we set
\begin{align*}
w_0=(a_1b_1)^{n_1} (a_2b_2)^{n_2} (a_1 b_1)^{n_3} (a_2b_2)^{n_4} \dots (a_sb_s)^{n_k},
\end{align*}
where $s$ is 1 or 2 depending on whether $k$ is odd or even. We set 
\begin{align*}
w= 
\begin{cases}
w_0 & \text{if } k \text{ is even},\\
w_0 (a_2b_2)^m & \text{if } k \text{ is odd}.
\end{cases}
\end{align*}
Then the $\Acode$ and $\Bcode$ of $w$ agree and are given by 
\begin{align*}
\Acode(w)= \Bcode(w) =
\begin{cases}
(n_1, \dots, n_k)=z & \text{if } k \text{ is even},\\
(n_1, \dots, n_k, m ) = (z,m) & \text{if } k \text{ is odd}.
\end{cases}
\end{align*} 

Since $m$ is distinct from all $n_i$, $m$ can never appear in any occurrence of $z$ or $\bar{z}$ in the $\Acode$ and $\Bcode$ of $w^\ell$. As in the first case, $\theta^A_z(w^\ell)= \theta^B_z(w^\ell)=\ell$, whereas $\theta^A_{\bar{z}}(w^\ell) = \theta^B_{\bar{z}}(w^\ell)=0$ since $z$ is generic. Consequently, 
\begin{align*}
f^A_z(w^\ell) + f^B_z(w^\ell)= \theta^A_z(w^\ell) + \theta^B_z(w^\ell) - \theta^A_{\bar{z}}(w^\ell) - \theta^B_{\bar{z}}(w^\ell)= 2 \ell,
\end{align*}
which shows that $f^A_z + f^B_z$ is unbounded and therefore its homogenisation is the desired unbounded Aut-invariant quasimorphism on $A \ast B$.

Finally, let us verify that the space of homogeneous Aut-invariant quasimorphisms on $A \ast B$ that vanish on letters is infinite-dimensional. Let $r \in \N$ and let $z_1, \dots, z_r$ be generic tuples. Choose $z_{r+1}$ be a 3-tuple whose entries are distinct non-zero natural numbers and do not appear in any of the $z_i$; then $z_{r+1}$ is generic. It follows from the above construction of the word $w$ for $z_{r+1}$ in both cases that any linear combination of the associated quasimorphisms $f^A_{z_1}+ f^B_{z_1}, \dots, f^A_{z_r}+ f^B_{z_r}$ vanishes on all powers of $w$. It follows that the same holds for any linear combination of their homogenisations $\bar{f}^A_{z_1}+ \bar{f}^B_{z_1}, \dots, \bar{f}^A_{z_r}+ \bar{f}^B_{z_r}$. Thus, $\bar{f}^A_{z_{r+1}}+ \bar{f}^B_{z_{r1}}$ is not contained in the subspace spanned by the first $r$ quasimorphisms. Clearly, the homogenisation of any code quasi-morphism vanishes on all letters of $A \ast B$. Since $r \in \N$ was arbitrary, it follows that the space of homogeneous Aut-invariant quasimorphisms on $A \ast B$ that vanish on letters cannot have finite dimension. 
\end{proof}

\section{Weighted code quasimorphisms}

If one of the factors of a free product $A \ast B$ of freely indecomposable groups happens to be infinite cyclic, the code quasimorphisms above are in general not Aut-invariant since they are not necessarily invariant with respect to transvections. Thus, we need to modify our original construction to deal with infinite cyclic factors. Afterwards we will follow steps similar to the previous section in order to establish their Aut-invariance.

\begin{lemma}\label{reduction letter}
Let $B$ be a non-trivial group and let $w$ be any word in $\Z \ast B$ such that $w$ only contains letters of the same sign from $\Z$ and starts and ends with a non-zero letter from $\Z$. Then its unique reduced form $w^\prime$ starts and ends with a letter from $\Z$ with that given sign. Moreover, the sum over all letters in $w$ belonging to the factor $\Z$ remains the same in its reduced form $w^\prime$.
\end{lemma}

\begin{proof} 
Any word in the free product is brought to its reduced form by successively eliminating trivial letters and replacing two adjacent letters from the same factor by their product in that factor.
The sum of all letters from $\Z$ stays the same because any two adjacent letters of $\Z$ are always replaced by their sum throughout the reduction process. The only way to encounter an elimination of the first letter $a_1 \in \Z$ or the last letter $a_n \in \Z$ during the reduction process would be by the occurrence of $-a_1$ or $-a_n$. This is not possible since $a_1$ and $a_n$ are non-zero and all letters have the same sign by assumption. 
\end{proof}

\begin{defi}[Weighted $\Z$-code]
Let $B$ be freely indecomposable and $B \ncong \Z$. Write $g \in \Z \ast B$ in reduced form. Let $(a_1, \dots , a_k)$ be the $\Z$-tuple of $g$. We define a tuple $(x_1, \dots , x_\ell)$ of non-zero natural numbers as follows. Consider the successive subsequences of maximal length in $(a_1, \dots ,a_k)$ consisting of integers all of the same sign. For the $i$-th such sequence, we define $x_i$ to be the absolute value of the sum of integers in that sequence. We call the tuple $(x_1, \dots , x_\ell)$ the \textit{weighted $\Z$-code} of $g$.
\end{defi}

\begin{example}
Let $B$ be a non-trivial group and let $b_i \in B$ be non-trivial elements. Then the reduced word 
\begin{align*}
w=7b_1 (-2) b_2 (-4) b_3 (-1) b_4 9b_5 2 b_6 (-3)
\end{align*} 
has the $\Z$-tuple $(7,-2.-4,-1,9,2,-3)$ which yields the weighted $\Zcode$ $(7,7,11,3)$.
\end{example}

\begin{defi}[Weighted code quasimorphisms]
Let $z=(n_1, \dots , n_k)$ be a tuple of non-zero natural numbers. We set $\theta^\Z_z \colon \Z \ast B \to \Z_{\geq 0}$ to count the number of \textit{disjoint} appearances of $z$ as a tuple of consecutive numbers inside the weighted $\Zcode$ of $g \in \Z \ast B$. Define the \textit{weighted code quasimorphism}  
\begin{align*}
f^\Z_z \colon \Z \ast B \to \Z \hspace{5mm} \text{by} \hspace{5mm} f^\Z_z(g)= \theta^{\Z}_{z}(g) - \theta^\Z_{z}(g^{-1})
\end{align*}
for all $g \in \Z \ast B$. Note that we again have $\theta^\Z_{z}(g^{-1})= \theta^\Z_{\bar{z}}(g)$ for all $g$. 
\end{defi}

\begin{lemma}\label{Code qm Z}
Let $z$ be a non-empty tuple of non-zero natural numbers. Then the counting function $\theta^\Z_z \colon \Z \ast B \to \Z_{\geq 0}$ satisfies 
\begin{enumerate}
\item $\theta^\Z_z(g^{-1})= \theta^\Z_{\bar{z}}(g)$ for all $g \in \Z \ast B$, 
\item $|\theta^\Z_z(w_1w_2)- \theta^\Z_z(w_1) - \theta^\Z_z(w_2) | \leq 2$ for all reduced words $w_1,w_2$ in  $\Z \ast B$ for which $w_1w_2$ is a reduced word. 
\end{enumerate}
Moreover, $f^\Z_z \colon \Z \ast B \to \Z$ defined for all $g \in \Z \ast B$ by  $f^\Z_z(g)= \theta^{\Z}_{z}(g) - \theta^\Z_{z}(g^{-1})$ is a quasimorphism of defect $D(f^\Z_z) \leq 30$.
\end{lemma}

\begin{proof}
First, recall that the reduced form of $g^{-1}$ is obtained by inverting the reduced form of $g$, which amounts to reversing the order and inverting all letters. This means to obtain the weighted $\Zcode$ of $g^{-1}$ one needs to reverse the one of $g$. Consequently, counting the number of disjoint occurrences of $z$ in the weighted $\Zcode$ of $g^{-1}$ amounts to counting the disjoint occurrences of the reversed tuple $\bar{z}$ in the weighted $\Zcode$ of $g$ itself. This proves the first part. 

Second, let $w_1, w_2$ be written as reduced words with $\Z$-tuples given by $(n_1, \dots, n_k)$ for $w_1$ and $(m_1, \dots, m_\ell)$ for $w_2$ for integers $n_i,m_j$. Let $(x_1, \dots, x_{k^\prime})$ and $(y_1, \dots , y_{\ell^\prime})$ be the weighted $\Zcode$s of $w_1$ and $w_2$. By assumption there is no cancellation or reduction in the product of their reduced expressions representing $w_1w_2$. That means the last letter of $w_1$ and the first letter of $w_2$ belong to different factors. Then  
\begin{align*}
\text{weighted } \Zcode(w_1w_2)= 
\begin{cases}
(x_1, \dots, x_{k^\prime}, y_1, \dots, y_{\ell^\prime}) & \text{if } \sgn(n_k) \neq \sgn(m_1), \\
(x_1, \dots, x_{k^\prime -1}, x_{k^\prime}+y_1,y_2, \dots, y_{\ell^\prime}) & \text{if } \sgn(n_k)=\sgn(m_1).
\end{cases}
\end{align*}

If $\sgn(n_k) \neq \sgn(m_1)$, then $\theta^{\Z}_z(w_1w_2 ) \in \{ \theta^\Z_z(w_1) + \theta^\Z_z(w_2), \theta^\Z_z(w_1) + \theta^\Z_z(w_2)+1 \}$ since at most one of the disjoint occurrences of $z$ can involve numbers that do not lie completely in the weighted $\Zcode$ of either $w_1$ or $w_2$. 

If $\sgn(n_k) = \sgn(m_1)$, then $\theta^\Z_z(w_1w_2 ) \geq \theta^\Z_z(w_1) + \theta^\Z_z(w_2) - 2$ since only one occurrence of $z$ in the weighted $\Zcode$ of $w_1$ and $w_2$ can involve the first or last number respectively. Moreover, if an occurrence of $z$ in the weighted $\Zcode$ of $w_1w_2$ involves $x_{k^\prime}+y_1$, then all other occurrences are fully contained in the weighted $\Zcode$ of either $w_1$ or $w_2$. Thus, $\theta^\Z_z(w_1w_2 ) \leq \theta^\Z_z(w_1) + \theta^\Z_z(w_2) +1 $. 

In both cases we conclude that
\begin{align*}
| \theta^\Z_z(w_1w_2 ) - \theta^\Z_z(w_1) - \theta^C_z(w_2)| \leq 2.
\end{align*}

It follows from Lemma \ref{building quasimorphisms} that $f^\Z_z$ is a quasimorphism of defect at most $30$.
\end{proof}

\begin{lemma}\label{Z-qm}
For all non-empty tuples $z$ the weighted code quasimorphism $f^\Z_z \colon \Z \ast B \to \Z$ is invariant under factor automorphisms and transvections. 
\end{lemma}

\begin{proof}
It is immediate from the definition that the weighted $\Zcode$ of any element in the free product is invariant under factor automorphisms. Let $x$ be a generator of the infinite cyclic factor in $\Z \ast B$. Any transvection is defined to be the identity on letters from $B$ and maps $x \to xy$ or $x \to yx$ for some non-trivial element $y \in B$. Let us consider the transvection $\varphi$ uniquely specified by $x \to xy$ and show that the weighted $\Zcode$ of any element in  $\Z \ast B$ is invariant under $\varphi$. Then it immediately follows that $\theta^\Z_z$ and $f^\Z_z$ are invariant under $\varphi$. The argument for transvections of the second kind will follow analogously to the one we present now. 

Let  $w \in \Z \ast B$ be a reduced word such that its weighted $\Zcode$ has length one. This means that all letters from $\Z$ in the reduced expression of $w$ have the same sign and the weighted $\Zcode$ is given by the image of $w$ under the factor projection $\Z \ast B \to \Z$. Note that this factor projection is invariant with respect to $\varphi$ and so the weighted $\Zcode$ of $\varphi(w)$ agrees with the one of $w$. There cannot be any cancellations of letters from $\Z$ occurring. 

Let us do a preliminary calculation to visualise the general case more easily. Let $k,\ell$ be non-zero natural numbers and $b \in B$ non-trivial. Then 
\begin{align*}
\varphi(x^k b x^{-\ell}) & = \varphi(x)^k b \varphi(x)^{-\ell} = (xy)^k b (y^{-1} x^{-1} )^\ell = xy \dots xyxyby^{-1} x^{-1} y^{-1} x^{-1} \dots y^{-1} x^{-1}, \\
\varphi(x^{-k} b x^{\ell}) & = \varphi(x)^{-k} b \varphi(x)^{\ell} = (y^{-1} x^{-1} )^k b (xy)^\ell = y^{-1} x^{-1} \dots y^{-1} x^{-1} b xy \dots xy.
\end{align*}
This shows that the letter from $B$ separating the positive and negative powers of $x$ either remains $b$ or is a conjugate of $b$ in $B$ after applying $\varphi$.

Let $w \in \Z \ast B$ be a reduced word with weighted $\Zcode$ of length $k \geq 2$. In  $w$ we formally gather all consecutive occurrences of powers of $x$ of the same sign and call these sub-words $w_i$ for $i= \{1, \dots, k \}$. That is, we write the reduced word $w$ uniquely as a product of reduced words as 
\begin{align*}
w=w_1 b_1 w_2 b_2 \dots w_{k-1} b_{k-1}w_k,
\end{align*}
where the $b_i \in B$ are non-trivial and the $w_i$ are of of maximal length such that all letters from $\Z$ inside any $w_i$ have the same sign. Moreover, in this decomposition $w_1$ ends with a letter from $\Z$, $w_n$ starts with a letter from $\Z$ and all other $w_i$ start and end with letters from $\Z$. By the maximality of $w_i$ all letters from $\Z$ occurring in $w_i$ have different signs from the ones occurring in $w_{i+1}$ for all $i$.

We apply $\varphi$ to $w$ and obtain an a priori not necessarily reduced word, which we rewrite in the previous block form as
\begin{align*}
\varphi(w)= \varphi(w_1) b_1 \varphi(w_2) b_2 \dots \varphi(w_{k-1})b_{k-1} \varphi(w_k) = w_1^\prime b_1^\prime w_2^\prime b_2^\prime \dots w_{k-1}^\prime b_{k-1}^\prime w_k^\prime, 
\end{align*}
where $b_i^\prime=yby^{-1}$ if the letters from $\Z$ change sign from positive to negative at $b_i$ and $b_i^\prime= b_i$ if they change from negative to positive. Moreover, all letters from $\Z$ inside any $w_i^\prime$ have the same sign again, $w_1^\prime$ ends with a letter from $\Z$, $w_n^\prime$ starts with a letter from $\Z$ and all other $w_i^\prime$ start and end with letters from $\Z$.

We observe that when bringing $\varphi(w)$ to its reduced form there cannot be any cancellations of the letters $b_i^\prime$. This is because by Lemma \ref{reduction letter} the letters that are adjacent to $b_i$ will always remain letters from $\Z$ after the reduction procedure of all $w_i^\prime$. Indeed, replacing all $w_i^\prime$ by their reduced forms $w_i^{\prime \prime}$ we see that the product 
\begin{align*}
w^{\prime \prime} = w_1^{\prime \prime} b_1^\prime w_2^{\prime \prime} b_2 \dots w_{k-1}^{\prime \prime} b_{k-1}^\prime w_k^{\prime \prime}
\end{align*}
is the reduced representative of $\varphi(w)$ since the letters adjacent to the $b_i^\prime$ are always letters from $\Z$. Consequently, no cancellations in between letters of different signs from $\Z$ can occur when bringing $\varphi(w)$ to its reduced form. The reduced words $w_i^{ \prime \prime}$ have the same weighted $\Zcode$ as the original $w_i$ for all $i$. Therefore, the weighted $\Zcode$ of $\varphi(w)$ agrees with the weighted $\Zcode$ of $w$.
\end{proof}

\begin{theorem} \label{freeprod with Z}
Let $B$ be a freely indecomposable group which is not infinite cyclic. Then for any generic tuple of natural numbers $z$ the homogenisation $\bar{f}^\Z_z \colon \Z \ast B \to \R$ of the quasimorphism $f^\Z_z$ is an unbounded Aut-invariant quasimorphism on $\Z \ast B$. Moreover, the space of homogeneous Aut-invariant quasimorphisms on $\Z \ast B$ that vanish on letters has infinite dimension. 
\end{theorem}

\begin{proof}
By Lemma \ref{Code qm Z} $f^\Z_z$ is a quasimorphism, which is invariant under factor automorphisms and transvections according to the previous Lemma \ref{Z-qm}. Images of these automorphisms generate the outer automorphism group $\Out(\Z \ast B)$ by Lemma \ref{outrepresentatives} . Thus, $f^\Z_z$ is invariant under a full set of representatives of all outer automorphisms and so by Lemma \ref{homogenisation Aut invariant} the homogenisation $\bar{f}^\Z_z$ is invariant under $\Aut(\Z \ast B)$. It remains to check that it is unbounded, which is equivalent to $f^\Z_z$ itself being unbounded. 

Since $z$ is generic, $z=(n_1, \dots , n_k)$ for some $k \geq 3$ where all $n_i \in \N$ are non-zero. Let $b \in B$ be non-trivial and $m$ a strictly positive integer number distinct from all $n_i$. Set $w \in \Z \ast B$ to be 
\begin{align*}
w = 
\begin{cases}
n_1 b (-n_2) b n_3 b (-n_4) \dots b (-n_k) b & \text{if } k \text{ is even},\\
n_1 b (-n_2) b n_3 b (-n_4) \dots b (-n_k)b mb & \text{if } k \text{ is odd}.
\end{cases}
\end{align*}
The weighted $\Zcode$ of $w$ is given by 
\begin{align*}
\text{weighted } \Zcode(w)= 
\begin{cases}
(n_1, \dots, n_k)=z & \text{if } k \text{ is even},\\
(n_1, \dots, n_k, m ) = (z,m) & \text{if } k \text{ is odd}.
\end{cases}
\end{align*}

Since $w$ starts and ends with letters belonging to different factors, the reduced expression of $w^\ell$ is the $\ell$-fold product of the word $w$ for all $\ell \in \N$. Moreover, since the first and last letter from $\Z$ in $w$ have different signs the weighted $\Zcode$ of $w^\ell$ is 
\begin{align*}
\text{weighted } \Zcode(w^\ell)= 
\begin{cases}
(z, z, \dots, z) & \text{if } k \text{ is even},\\
(z, m, z, m, \dots, z, m) & \text{if } k \text{ is odd}.
\end{cases}
\end{align*}
Since $m$ is distinct from all $n_i$, $m$ cannot appear in any occurrence of $z$ or $\bar{z}$ inside the weighted $\Zcode$ of $w^\ell$. So $\theta^\Z_z(w^\ell)= \ell$, whereas $\theta^A_{\bar{z}}(w^\ell)=0$ since $\bar{z}$ is $z$ is generic. Consequently, 
\begin{align*}
f^\Z_z(w^\ell)= \theta^\Z_z(w^\ell)- \theta^\Z_{\bar{z}}(w^\ell)= \ell,
\end{align*}
which shows that $f^\Z_z$ is unbounded.

Finally, let us verify that the space of homogeneous Aut-invariant quasimorphisms on $\Z \ast B$ that vanish on letters is infinite-dimensional. Let $r \in \N$ and let $z_1, \dots, z_r$ be generic tuples. Choose $z_{r+1}$ be a 3-tuple whose entries are distinct non-zero natural numbers and do not appear in any of the $z_i$; then $z_{r+1}$ is generic. It follows from the above construction of the word $w$ for $z_{r+1}$ that any linear combination of $f^\Z_{z_1}, \dots, f^\Z_{z_r}$ vanishes on all powers of this $w$. It follows that the same holds for any linear combination of their homogenisations $\bar{f}^\Z_{z_1}, \dots, \bar{f}^\Z_{z_r}$. Thus,  $\bar{f}^\Z_{z_{r+1}}$ is not contained in the subspace spanned by the first $r$ quasimorphisms.  Clearly, the homogenisation of any weighted code quasi-morphism vanishes on all letters of $\Z \ast B$. Since, $r \in \N$ was arbitrary, it follows that the space of homogeneous Aut-invariant quasimorphisms on $\Z \ast B$ that vanish on letters cannot have finite dimension. 
\end{proof}

\section{Applications of code quasimorphisms}

\begin{proof}[Proof of Theorem \ref{T1}]
By \cite[Theorem 2]{BraMarc} the space of homogeneous Aut-invariant quasimorphisms on $\Z \ast \Z$ is infinite-dimensional. Inverting both generators of the factors defines an automorphism which inverts all letters in $\Z \ast \Z$. So any homogeneous Aut-invariant quasimorphism on $\Z \ast \Z$ vanishes on all letters. For all other free products of two factors Proposition \ref{freeprod no Z prop} and Proposition \ref{freeprod with Z} imply the existence of infinitely many linearly independent homogeneous Aut-invariant quasimorphisms, all of which vanish on letters. 
\end{proof}

\begin{proof}[Proof of Corollary \ref{MainCor}]
Let $A \ast B$ be a free product of two freely indecomposable groups which is not the infinite dihedral group. By Theorem \ref{T1} there exist unbounded Aut-invariant quasimorphisms on $A \ast B$ that are bounded on all letters. Since $A \ast B$ is generated by letters, the result follows from Lemma \ref{Autqm implies Autnorm}.
\end{proof}

\begin{remark}
If neither $A$ nor $B$ is infinite cyclic, then Corollary \ref{MainCor} can also be deduced from the result given in \cite[Lemma 4.4]{Marcinkowski} together with the explicit description of the automorphism group given in Section 3. 
\end{remark}

\begin{cor}\label{more factors qm}
Let $G = \ast_{i \in I} G_i$ be a free product of finitely many freely indecomposable groups $G_i$. Assume there exist free factors $G_j$ and $G_k$ with $j \neq k$ such that no free factors $G_i$ for $i \notin \{ j,k \}$ is isomorphic to $G_j$ or $G_k$ or is infinite cyclic. Moreover, assume that $G_j$, $G_k$ are not both equal to $\Z/2$. Then any unbounded Aut-invariant quasimorphism on $G_j \ast G_k$ gives rise to an unbounded Aut-invariant quasimorphism on $G$. In particular, the space of homogeneous Aut-invariant quasimorphisms on $G$ is infinite-dimensional. 
\end{cor}

\begin{proof}
We claim that the projection $p \colon G \rightarrow G_j \ast G_k$ is Aut-equivariant, i.e. any automorphism of $G$ descends via $p$ to an automorphism of $G_j \ast G_k$. This is equivalent to $\ker(p)$ being a characteristic subgroup of $G$. Once this is established, we apply Theorem \ref{T1} to $G_j \ast G_k$ and conclude the proof by applying Lemma \ref{Gen6}.

Let us now show that any automorphism of $G$ indeed descends to $G_j \ast G_k$. By \cite{Gilbert} $\Aut(G)$ is generated by factor automorphisms, swap automorphisms, partial conjugations and transvections since inner automorphisms can be written as products of factor automorphisms and partial conjugations. It is clear that all factor automorphisms and all partial conjugations of $G$ descend to automorphisms of $G_j \ast G_k$ via $p$. By our assumption there are no swap automorphisms permuting any other free factors in $G$ with $G_j$ and $G_k$, so these descend to the quotient as well. It only remains to check the transvections if $G_j$ or $G_k$ happen to be infinite cyclic.  So let $G_j$ be infinite cyclic generated by $x$ and let $a$ be a letter from a different factor $G_\ell$. If $\ell=k$, then any transvection $\varphi_a$ defined by $\varphi_a(x)=ax$ or $\varphi_a(x)=xa$ descends via $p$ to the same transvection on $G_j \ast G_k$. If $\ell \neq k$, any such transvection descends to the identity on $G_j \ast G_k$. In particular, it always descends via $p$. Since a generating set of $\Aut(G)$ descends to automorphisms of the quotient $G_j \ast G_k$, any element of $\Aut(G)$ does so. Consequently, the map $p$ is Aut-equivariant. 
\end{proof}

\begin{cor} \label{extensionchar}
Let $H \to G \to A \ast B$ be an extension of a free product of freely indecomposable groups $A$ and $B$ by a group $H$. Assume that $H$ is a characteristic subgroup of $G$ and $A \ast B$ is not the infinite dihedral group. Then the space of homogeneous Aut-invariant quasimorphisms on $G$ is infinite-dimensional.   
\end{cor}

\begin{proof}
The space of homogeneous Aut-invariant quasimorphisms on $A \ast B$ is infinite-dimensional by Theorem \ref{T1}. Therefore, the result follows from Lemma \ref{Gen6}.
\end{proof}

\begin{cor}\label{amalgam}
Let $G_1 \ast_H G_2$ be a free product of groups $G_1,G_2$ amalgamated over a common subgroup $H$ which is proper and central in both $G_1$ and  $G_2$. If $G_1/H$ and $G_2/H$ are freely indecomposable and not both equal to $\Z/2$, the space of homogeneous Aut-invariant quasimorphisms on $G_1 \ast_H G_2$ is infinite-dimensional. 
\end{cor}

\begin{proof}
By assumption $H \neq G_1$ and $H \neq G_2$ and so $H$ equals the center of $G_1 \ast_H G_2$. As such it is a characteristic subgroup of $G_1 \ast_H G_2$. Furthermore, $\frac{G_1 \ast_H G_2}{H} \cong \frac{G_1}{H} \ast \frac{G_2}{H}$. By assumption $\frac{G_1}{H} \ast \frac{G_2}{H}$ is not isomorphic to the infinite dihedral group and the result follows from Corollary \ref{extensionchar} above.
\end{proof}

\begin{example}
For $q \geq 3$, the Hecke groups $H_q \cong \Z/2 \ast \Z/q$ admit infinitely many linearly independent homogeneous Aut-invariant quasimorphisms by Theorem \ref{T1}. 
\end{example}

\begin{example}
By Corollary \ref{amalgam} the space of homogeneous Aut-invariant quasimorphisms on $\SL(2,\Z)$ is infinite-dimensional, since $\SL(2,\Z)$ is the amalgamated product $\Z/4 \ast_{\Z/2} \Z/6$.   
\end{example}

\begin{example}
The braid group $B_3$ admits infinitely many linearly independent homogeneous Aut-invariant quasimorphism as well by Corollary \ref{extensionchar}. Indeed, $B_3$ is the universal central extension of $\PSL(2,\Z)= \Z/2 \ast \Z/3$ by $\Z$. 
\end{example}

\begin{example}
Let $G_{p,q} = \Z \ast_\Z \Z$ be the free product of two copies of the integers amalgamated over inclusions $\iota_1 , \iota_2 \colon \Z \to \Z$ which are multiplication by $p$ and $q$. For coprime choices of $p$ and $q$ these are the so called knot groups $K_{p,q}$ arising as the fundamental group of the complement of torus knots. Then $G_{p,q}$ admits infinitely many linearly independent homogeneous Aut-invariant quasimorphisms if $\min\{ |p|, |q| \} \geq 2$ and $\max \{ |p|,|q| \} \geq 3$. We have seen in Example \ref{Kleinbottle} that this is no longer true for $p=q=2$.
\end{example}

\begin{example}
$B_3 \ast_\Z B_3$, the free product of $B_3$ with itself amalgamated over their common center generated by the Garside element, admits infinitely many linearly independent unbounded Aut-invariant quasimorphisms. To prove this we cannot apply Corollary \ref{amalgam} directly since $B_3/\Z$ is not freely indecomposable. The center of $B_3 \ast_\Z B_3$ is again generated by the Garside element of each of the factors. This fits into the short exact sequence 
\begin{align*}
\Z \to B_3 \ast_\Z B_3 \to (B_3/\Z) \ast (B_3/\Z).
\end{align*}
Finally, $(B_3/\Z) \ast (B_3/\Z)= \PSL(2,\Z) \ast \PSL(2,\Z)= \Z/2 \ast \Z/3 \ast \Z/2 \ast \Z/3$. So Corollary \ref{more factors qm} applies with $G_j=G_k=\Z/3$. Then the statement for $B_3 \ast_\Z B_3$ follows from Lemma \ref{Gen6}.
\end{example}

\section{Aut-invariant stable commutator length}

For any group $G$ let $\cl_G$ denote the commutator length on $[G,G]$, which is defined to be the minimal number of commutators required to write a given element of the commutator subgroup. Let $\scl_G(x) = \lim_n \frac{\cl(x^n)}{n}$ denote the stable commutator length of $x \in [G,G]$. It shares a deep relationship with quasimorphisms on $G$ through the so called Bavard duality \cite{Calegari}. We now define the Aut-invariant (stable) commutator length; this is a special case of the $\hat{G}$-invariant (stable) commutator length defined in \cite{KK}. 

\begin{defi}
Let $G \leq \hat{G}$ be a normal subgroup. Consider the subgroup $[\hat{G},G] \leq G$ generated by commutators of the form $[F,g]$ and their inverses where $F \in \hat{G}$ and $g \in G$. Then for $x \in [\hat{G},G]$ the \textit{$\hat{G}$-invariant commutator length} $\cl_{\hat{G},G}(x)$ is defined to be the minimal length of an expression of $x$ as a product of commutators $[F,g]$ and their inverses where $F \in \hat{G}$ and $g \in G$. The  \textit{$\hat{G}$-invariant stable commutator length} $\scl_{\hat{G},G}$ for $x \in [\hat{G},G]$ is defined by $\scl_{\hat{G},G}(x) = \lim_n \frac{\cl_{\hat{G}G}(x^n)}{n}$. 

Given any group $G$, its inner automorphism group $\Inn(G)$ is a normal subgroup of $\Aut(G)$ and so the above definition applies to $\Inn(G)$. If $G$ has trivial center, $G$ can be identified with $\Inn(G)$. In this case we simplify the notation by denoting the \textit{Aut(G)-invariant commutator length} simply as $\cl_{\Aut}$ and the \textit{Aut(G)-invariant stable commutator length} simply as $\scl_{\Aut}$.
\end{defi}

Setting $\hat{G}=\Aut(G)$ the following lemma is proven in \cite[Lemma 2.1]{KK}.

\begin{lemma} \label{Kawasaki Lemma}
Let $G$ be a group with trivial center so that $G= \Inn(G)$. Let $\phi$ be an homogeneous Aut-invariant quasimorphism on $G$. Then any $x \in [\Aut(G),G]\leq G$ satisfies 
\begin{align*}
\scl_{\Aut}(x) \geq \frac{1}{2}\frac{|\phi(x)|}{D(\phi)}. 
\end{align*} \qed
\end{lemma} 

In fact, according to \cite[Theorem 1.3]{KK} $\hat{G}$-invariant quasimorphisms satisfy an analogue of the Bavard duality theorem if $[\hat{G},G] =G$. All free products $A \ast B$ of freely indecomposable groups $A$ and $B$ have trivial center and so the notions $\cl_{\Aut}$ and $\scl_{\Aut}$ apply. However, free products often fail to satisfy $[\Aut(G),G] =G$. We will use a constructive approach rather than relying on an invariant analogue of Bavard's duality in the following.

\begin{example}
For $D_\infty = \Z/2 \ast \Z/2$ it holds that $\scl_{\Aut} \equiv 0$. To see this denote the generators of the $\Z/2$ factors by $a$ and $b$ and let $s$ be the automorphism of $D_\infty$ interchanging $a$ and $b$. Then we calculate $[s,a]=sas^{-1}a^{-1}=s(a)a^{-1}=ba$ and $[s,b]=ab=(ba)^{-1}$. Since the total number of letters appearing in any expression of the form $[f,x]$ is always even where $f \in \Aut(D_\infty)$ and $x \in D_\infty$, it holds that $[\Aut(D_\infty),D_\infty)] \cong \Z$ generated by $ba$. In fact, any power $(ba)^k$ for $k \in \Z$ can be written as a single commutator $[s,w]$, where $w$ is one of the two words of length $|k|$. Thus, $\cl_{\Aut}$ is equal to one for any non-trivial element in $[\Aut(D_\infty),D_\infty)]$ and $\scl_{\Aut}$ vanishes. 
\end{example}

\begin{example}
Consider $G = \PSL(2,\Z)= \Z/3 \ast \Z/2$. Then $\Aut(G)$ is generated by the set $C$ consisting of  the non-trivial factor automorphism of $\Z/3$ and conjugations by letters of $\Z/3$ and $\Z/2$, since both free factors are abelian groups. Consequently, $[\Aut(G),G]$ is normally generated by commutators of the form $[c,g]=cgc^{-1}g^{-1}=c(g)g^{-1}$ for $c \in C$ and $g \in G$. In all expressions $[c,g]$ the letter $b$ representing the non-trivial element of the factor $\Z/2$ arises an even number of times. Therefore, $b \notin [\Aut(G),G]$ and the latter is not the full group $G$. 
\end{example}

\begin{lemma} \label{Aut finite index}
Let $G=A \ast B$ be a free product of freely indecomposable groups where at least one of the factors is infinite cyclic. Then $[\Aut(G),G]$ has index at most 2 in $G$. Therefore, any unbounded quasimorphism on $G$ is unbounded when restricted to $[\Aut(G),G]$.
\end{lemma}

\begin{proof}
If an unbounded quasimorphism $q$ is bounded on a finite index subgroup $H \leq G$, its homogenisation $\bar{q}$ vanishes on $H$. Then $\bar{q}$ descends to a map of sets on the finite set $G/H$ implying that the image of $\bar{q}$ is bounded and so was the image of $q$ to begin with. Therefore, any quasimorphism $q$ with unbounded image cannot be bounded on $H$. So, it remains to show that the index of $[\Aut(G),G]$ in $G$ is finite to prove the lemma. 

First, consider the case where $A$ and  $B$ are both infinite cyclic and so $G$ can be identified with $F_2$, the free group of rank 2. Let $x$ and $y$ be standard generators. Consider the automorphism $\varphi$ of $F_2$ defined by $\varphi(x)=yx$ and $\varphi(y)=y$. Then $[\varphi, x]=\varphi(x) \cdot x^{-1} = y$. So $\langle y \rangle \leq [\Aut(F_2),F_2]$. By symmetry of the generating set it holds that $\langle x \rangle \leq [\Aut(F_2),F_2]$ as well and it follows that $[\Aut(F_2),F_2]=F_2$. 

Second, consider the case where only one of the factors is infinite cyclic. Without loss of generality assume $A= \Z$. Let $x$ denote a generator of $A$. For any $b \in B$ we can define the transvection $\varphi_b$ on $x$ by $\varphi_b(x)=bx$ and by $\varphi_b(b^\prime)=b^\prime$ for all $b^\prime \in B$. Then $\varphi_b$ is an automorphism and satisfies $[\varphi_b,x]= \varphi_b(x) \cdot x^{-1} =b$ for all $b \in B$. Thus, $B \leq [\Aut(G),G]$. Moreover, denoting the non-trivial factor automorphism of $A= \Z$ by $f$ we compute $[f,x]=f(x) \cdot x^{-1} = x^{-2}$ and deduce that $2 \Z \leq [\Aut(G),G]$. Therefore, 
$2 \Z \ast B \leq [\Aut(G),G]$. Since $[\Aut(G),G]$ is a normal subgroup, it holds that $N \leq [\Aut(G),G]$ where $N$ is the normal closure of $2 \Z \ast B$. However, $G/N \leq \Z/2$ and so $[\Aut(G),G]$ has at most index 2 in $G$. 
\end{proof}

\begin{proof}[Proof of Theorem \ref{T2}]
If one of the factors is infinite cyclic, Theorem \ref{T1} implies the existence of an unbounded Aut(G)-invariant homogeneous quasimorphism, which is unbounded on $[\Aut(G),G]$ according to Lemma \ref{Aut finite index}. The statement then follows from Lemma \ref{Kawasaki Lemma}. 

Assume from now on that neither $A$ nor $B$ is infinite cyclic. We will prove the theorem by explicitly constructing an element in $[\Aut(G),G]$ together with homogeneous Aut-invariant quasimorphism which is non-trivial on that element. 

The only non-trivial group with no non-trivial automorphisms is $\Z/2$. Thus, one of the factors has to have a non-trivial automorphism since $A \ast B$ is not the infinite dihedral group. Without loss of generality we assume that $|\Aut(A)|\geq 2$. Let $a_1 \in A$ such that $f(a_1)=a_2 \neq a_1$ for some $f \in \Aut(A)$. Then we compute that $[f,a_1]=a_2 a_1^{-1}$ is a non-trivial element in $[\Aut(G),G]$. Since $[\Aut(G),G]$ is normal, it holds for any $h \in B$ that $h a_2 a_1^{-1} h^{-1} \in [\Aut(G),G]$. Thus, $(a_2 a_1^{-1}h a_2 a_1^{-1} h^{-1})^k \in [\Aut(G),G]$ for any $k \in \N$ and $h \in B$. 

First, assume that $A$ and $B$ are not isomorphic. Since $|\Aut(A)|\geq 2$, it holds that $|A| \geq 3$ and we can choose a non-trivial $a \neq a_2 a_1^{-1}$ and fix some non-trivial $h \in B$. We define for $n_1, \dots, n_\ell \in \N$ the word 
\begin{align*}
w = \prod_{i=1}^\ell \big ( [a,h] (a_2 a_1^{-1}h a_2 a_1^{-1} h^{-1})^{n_i} \big ) \in [\Aut(G),G].
\end{align*}
We calculate 
\begin{align*}
\Acode(w) = 
\begin{cases}
(1,1,2n_1,1,1,2n_2, \dots, 1,1,2n_\ell) & \text{if} \hspace{1mm} a^{-1} \neq  a_2 a_1^{-1}, \\
(1,2n_1+1,1, 2n_2+1, \dots ,1,2n_\ell+1)  & \text{if} \hspace{1mm} a^{-1} =  a_2 a_1^{-1}.
\end{cases}
\end{align*}
Set $z = \Acode(w)$. For all $n \in \N$ it holds that $\Acode(w^n)=(z, \dots,z)$. Choose $\ell \geq 3$ together with large and distinct $n_1, \dots, n_\ell \in \N$, which implies that $z$ is generic. It follows from Proposition \ref{freeprod no Z prop} that $\bar{f}^A_z$ is an Aut-invariant quasimorphism. By construction $f^A_z$ satisfies $f^A_z(w^n)=n$ for all $n \in \N$ and so $\bar{f}^A_z(w)>0$ for our choice of $w \in [\Aut(G),G]$. The statement then follows from Lemma \ref{Kawasaki Lemma}. 

Second, assume $A \cong B$. Let $s$ denote a swap automorphism. Set $b_i = s(a_i)$ for $i \in \{1,2\}$. Then the element $[s f s^{-1}, b_1]= b_2b_1^{-1} \in [\Aut(G),G]$ is non-trivial and belongs to the factor $B$. Set $b=s(a)$. Then $a^{-1}= a_2 a_1^{-1}$ is equivalent to $b^{-1}=b_2 b_1^{-1}$. For $n_1, \dots, n_\ell \in \N$ we define the word
\begin{align*}
w = \prod_{i=1}^\ell \big ( [a,b] (a_2 a_1^{-1} b_2 b_1^{-1})^{n_i} \big ) \in [\Aut(G),G].
\end{align*}
Observe, that $\Acode(w)= \Bcode(w)$. As in the previous case, 
\begin{align*}
\Acode(w) = 
\begin{cases}
(1,1,n_1,1,1,n_2, \dots, 1,1,n_\ell) & \text{if} \hspace{1mm} a^{-1} \neq  a_2 a_1^{-1}, \\
(1,n_1+1,1, n_2+1, \dots ,1,n_\ell+1)  & \text{if} \hspace{1mm} a^{-1} =  a_2 a_1^{-1}.
\end{cases}
\end{align*}
Set $z = \Acode(w)$ as before and choose $\ell \geq 3$ and $n_1, \dots, n_\ell \in \N$ large enough and distinct, so that $z$ is generic. Again, we calculate $f^A_{z}(w^n)=n= f^B_{z}(w^n)$, which implies that $(\bar{f}^A_{z} + \bar{f}^B_{z})(w)>0$. Proposition \ref{freeprod no Z prop} implies that $\bar{f}^A_{z} + \bar{f}^B_{z}$ is an homogeneous Aut-invariant quasimorphism and so applying Lemma \ref{Kawasaki Lemma} concludes the proof.
\end{proof}

We will now give a few more examples of free products $G$ where $[\Aut(G),G]=G$ and therefore the notions of Aut-invariant (stable) commutator length are defined on all of $G$.  

\begin{example}
For a product $G=A \ast B$ of two freely indecomposable perfect groups $A$ and $B$ it holds that $[\Aut(G),G]=G$. Indeed, since $A$ perfect, it holds that $A= [A,A] \leq [\Aut(G),G]$. Similarly, $B \leq [\Aut(G),G]$ and $A$ and $B$ generate $G$ it follows that $[\Aut(G),G]=G$.
\end{example}

\begin{example}
$G= \Z/p \ast \Z/q$ satisfies $[\Aut(G),G]= G$ for $p,q \geq 3$ prime. Let $m \colon \Z/p \to \Z/p$ be multiplication by 2, which is a factor automorphism. Consider the standard generator $1_p \in \Z/p$. It holds that $[m,1_p]=m(1_p)-1_p=1_p$. Thus, $\Z/p \leq [\Aut(G),G]$. Similarly, $\Z/q \leq [\Aut(G),G]$ and so $[\Aut(G),G]= G$.
\end{example}

\begin{example}
Let $k, \ell \geq 2$. Then $G= A^k \ast B^\ell$ satisfies $[\Aut(G),G]= G$ for all non-trivial abelian groups $A$ and $B$. For simplicity of notation consider the case $k=2$. Let $\phi$ be the factor automorphism defined by $\phi(a,0)=(a,a)$  and $\phi(0,a)=a$ for all $a \in A$. Then $[\phi, (a,0)]=\phi(a,0)-(a,0)=(0,a)$ for all $a \in A$. Thus, $A \times 0 \leq [\Aut(G),G]$. Analogously, $0 \times A \leq [\Aut(G),G]$. Since these factors generate $A^2$ it follows that $A^2 \leq [\Aut(G),G]$. Similarly, $B^\ell \leq [\Aut(G),G]$ and so $[\Aut(G),G]= G$.
\end{example}

\begin{example}
In all of the above examples for $G=A \ast B$ the equality $[\Aut(G),G]=G$ is always derived by showing that both factors $A$ and $B$ form subgroups of $[\Aut(G),G]$. Thus, any combination of free factors appearing in the three examples above still satisfies this equality, for example $A=\Z/p$ for $p \geq 3$ prime and $B$ any freely indecomposable perfect group. 
\end{example}

\section*{Acknowledgements}

I like to thank Jarek K\k{e}dra and Benjamin Martin for their continued support and all their helpful comments. This work was partly funded by the Leverhulme Trust Research Project Grant RPG-2017-159.

\vspace{8mm}
\noindent
BASTIEN KARLHOFER, UNIVERSITY OF ABERDEEN, UK.\\
\noindent 
EMAIL: r01bdk17@abdn.ac.uk

%\end{thebibliography}
\end{document}